\documentclass[11pt]{amsart}
\usepackage{graphicx, amsmath, fullpage, amssymb, amsthm, amsfonts, color}
\usepackage{enumerate}
\newtheorem{theorem}{Theorem}[section]
\newtheorem{lemma}[theorem]{Lemma}

\theoremstyle{definition}
\newtheorem{definition}[theorem]{Definition}

\newtheorem{cor}[theorem]{Corollary}
\usepackage[toc,page]{appendix} 

\theoremstyle{remark}
\newtheorem{remark}[theorem]{Remark}
\numberwithin{equation}{section}
\usepackage[colorlinks,
            linkcolor=red,
            anchorcolor=red,
            citecolor=blue
            ]{hyperref}
\begin{document}


\title{Sparse General Wigner-type Matrices: Local Law and Eigenvector Delocalization} %Title of paper



\author{Ioana Dumitriu}
\email[]{dumitriu@uw.edu}
%\thanks{}
\address{Department of Mathematics, University of Washington}

\author{Yizhe Zhu}
\email[]{yizhezhu@uw.edu}
%\thanks{}
\address{Department of Mathematics, University of Washington}


\date{\today}

\begin{abstract}
We prove a local law and eigenvector delocalization for general Wigner-type matrices. Our methods allow us to get the best
possible interval length and optimal eigenvector delocalization in the dense case, and the first results of
such kind for the sparse case down to $p=\frac{g(n)\log n}{n}$ with $g(n)\to\infty$. We specialize our results to
the case of the Stochastic Block Model, and we also obtain a local law
for the case when the number of classes is unbounded.
\end{abstract}


\maketitle %\maketitle must follow title, authors, abstract and \pacs

\section{Introduction}

\subsection{The Stochastic Block Model}

The Stochastic Block Model (SBM), first introduced by mathematical
sociologists  \cite{HLL83}, is a widely used
random graph model for networks with communities. In the last decade,
there has been considerable activity \cite{coja2010graph,
krzakala2013spectral, abbe2016exact, abbe2015community,brito2016recovery,
BGBK1,BGBK2} in understanding the spectral
properties of matrices associated to the SBM and to other generalized
graph models, in particular in connection to
spectral clustering methods. 

Stochastic Block Models represent a generalization of Erd\H{o}s-R\'enyi graphs to allow for more heterogeneity. Roughly
speaking, an SBM graph starts with a partitioning of the vertices into
classes, followed by placing an Erd\H{o}s-R\'enyi graph on each class
(independent edges, each occurring with the same given probability
depending on the class), and connecting vertices in two different blocks by independent edges,
again with the same given probability which this time depends on the
pair of classes. The random matrix associated to
this graph is the adjacency matrix, which is a random block matrix
whose entries have Bernoulli distributions, the parameters of which are
dictated by the inter- and intra-block probabilities mentioned
above. 

Specifically, suppose for ease of numbering that $[n] = V_1 \cup V_2
\cup \ldots \cup V_d$ for some integer $d$, $|V_i|=N_i$   for $i = 1, \ldots, d$. Suppose that for any pair
$(k,l) \in [d] \times [d]$ with $k \neq l$ there is a $p_{kl} \in [0,1]$ such that for
any $i \in V_k$, $j \in V_l$, 
\[
a_{ij} = \left \{ \begin{array}{ll} 1,
    & \mbox{with probability } p_{kl}, \\
0, & \mbox{otherwise}. \end{array} \right . 
\]
Also, if $k=l$, there is a $p_{k}$ such that for any $i, j \in V_k$, 
\[
a_{ij} = \left \{ \begin{array}{ll} 0, & \mbox{if } i=j, \\
                    1, & \mbox{with probability $p_k$,} \\
                    0, & \mbox{otherwise}. \end{array} \right .
\]
Each diagonal block is an adjacency matrix of a simple  Erd\H{o}s-R\'enyi graph  and  off-diagonal blocks are adjacency matrix of bipartite graphs. 
While there is interest in studying the $O(1)$ variance case
(corresponding to all $p_{ij}$s and $p_{i}$s being $O(1)$; the
``dense'' case), special interest is given to the
sparse case (when $p_{ij}$s and $p_{i}$s are $o(1)$, and more
specifically, when the average vertex degrees, given by $np_{ij}$ as well as $np_i$, are growing very slowly
with $n$ or may even be large and constant).  

The adjacency matrices of SBM graphs are themselves a particular form of general
Wigner-type matrix, which have been shown to exhibit universal
properties in \cite{ajanki2015universality} in the dense case. We detail the connection to
the broader field of random matrix universality studies in the next
section. 

\subsection{Universality studies, general Wigner-type matrices, and
  related graph-based matrix models}

At the same time with the increased interest in the spectra of SBM, the universality studies in random matrix theory
pioneered by  \cite{tao2011random} and 
\cite{erdos2009} had been gaining ground at tremendous pace. The pioneering
work on Wigner matrices started in \cite{tao2011random} and \cite{erdos2009} has been now
extended to cover generalized Wigner matrices \cite{erdos2012bulk}, Erd\H{o}s-R\'enyi matrices (including
sparse ones) \cite{er_univer1, er_univer2, erdos2011universality}, and general Wigner-type matrices \cite{ajanki2015universality}. All
such studies start by proving a local law at ``optimal'' scale, that is, on
intervals of length $(\log n)^{\alpha}/n$ or $n^{1-\varepsilon}$, which
is necessary for the complicated machinery of either \cite{tao2011random} or
 \cite{erdHos2010bulk} to translate the local law into universality of
eigenstatistics on ``optimal-length'' intervals.

In this paper we prove a main theorem about (dense) generalized Wigner
matrices and then apply it to cover sparse generalized Wigner
matrices; finally, we show that our results translate to graph models
like the SBM with bounded or unbounded number of blocks. We provide below a brief review of universality studies related to graph-based models.

After the original work on Wigner matrices, the first step in the
direction of graph models, or graph-based matrices, came with \cite{tran2013sparse},
where the authors proved a local law for Erd\H{o}s-R\'enyi graphs.
Subsequently, \cite{er_univer1}, \cite{er_univer2} superseeded these results for the
slightly denser cases and showed bulk universality in a $p \gg n^{-1/3}$
regime. The sparsity of the model is important here, because it makes the problem more
difficult. 
x
In a more recent paper, \cite{huang2015bulk} refined the results of
\cite{er_univer1, er_univer2} and made
them applicable for $p \gg n^{-1+\varepsilon}$ for any (fixed)
$\varepsilon>0$. Subsequently, in a departure from studying adjacency
matrices, \cite{er_laplacians} proved bulk universality
for the eigenvalue statistics of Laplacian matrices of
Erd\H{o}s-R\'enyi graphs, in the regime when $p \gg  n^{-1+\varepsilon}$
for fixed $\varepsilon>0$.  Very recently, \cite{he2018local} proved a local law for the adjacency matrix of the  Erd\H{o}s-R\'enyi graph with $p\geq C\log n/n$ for some constant $C>0$.

Finally, \cite{adlam2015spectral} examined a large class of sparse random graph-based matrices (two-atom and three-atom
entry distributions), proved a local law up to intervals of length
$1/n^{1-\varepsilon}$ and deduced (by the same means employed in \cite{huang2015bulk})
a bulk universality theorem. This is different from our results, since
our sparse matrices have entries that are not necessarily atomic, but
come from the product of a Bernoulli variable and a potentially
continuous one (See Section \ref{SparseRM}); however, the \cite{adlam2015spectral} case does seem
to cover the sparse SBM model for $p\gg n^{-1+\varepsilon}$ . As an
interesting aside, in the general case there may not be an asymptotic global law (aka the
limiting empirical spectral distribution); the cases we study here (SBM with bounded and unbounded
number of classes) are specific enough that
we can also prove the asymptotic global law. However, as it turns out, in
the case of the SBM with an unbounded number of blocks, the prediction
in the local law must still be made using the $n$-step approximation
to the global law, not the global law itself, since convergence to the
global law is not uniform.

Some of the methods used for examining the universality of these graph
models rely on the work in \cite{ajanki2015universality, ajanki2015quadratic}, where a general
(dense) Wigner-type matrix model is considered and universality is
proved up to intervals of length $1/n^{1-\varepsilon}$. We will also
appeal to  \cite{ajanki2015quadratic}, since it will help establish the
existence of limiting distributions and the stability of their Stieltjes transforms. 

We should mention the significant body of literature that
deals with global limits for the empirical spectral distributions of block
matrices. Starting with the seminal work of Girko \cite{girko2001theory}, the topic
was treated in \cite{far2006spectra} and \cite{shlyakhtenko1998gaussian} from a free
probability perspective; more recently, \cite{avrachenkov2015spectral}
and \cite{ding2014spectral} have examined the topic again for finitely many blocks (a claim in
\cite{ding2014spectral} that the method extends to a growing number of blocks is
incorrect). The global law for stochastic block models with a growing number of blocks was derived in \cite{zhu2018graphon} via  graphon theory.

\subsection{This paper}
The main difference in the results streaming from the seminal works of 
\cite{tao2011random}, respectively, \cite{erdos2009}, is in the conditions imposed on the
matrix entries: the former approach to universality is based on the
``four moment match'' condition, but imposes relatively weak
conditions on the tails, while the latter works by imposing
stronger conditions on the tails. In later works, these stronger conditions have
included bounded moments \cite{erdos2012bulk, huang2015bulk,
  ajanki2015universality}. While the methods of
\cite{erdos2009} have been extended to increasingly more general matrix models,
the methods in \cite{tao2011random} have been used to focus on reaching the best (smallest) possible
interval lengths for the classical, Wigner case via methods whose
basis was set in \cite{vu2015random}. 

This paper bridges the two approaches to obtain a local law and
eigenvector delocalization for dense and sparse general Wigner-type matrices 
and for the SBM. 

Our main result is a local law in the bulk down to interval
length $\frac{CK^2\log n}{n}$ for general Wigner-type matrices (see Section \ref{GWM}) whose entries are compactly supported almost surely in $[-K,K]$, employing some of the ideas from \cite{tao2011random} and
\cite{vu2015random}. Our result is more refined than the one from
\cite{ajanki2015universality}, where the smallest interval length was
$O(1/n^{1-\varepsilon})$ and bounded moments were assumed. With
additional assumptions (either four-moment matching, as in the case of
\cite{tao2011random}, or finite moments, as in \cite{erdos2009} and subsequent works) universality
down to this smaller interval length should follow.

In addition to this main result, we also obtain the first local laws
for sparse general Wigner-type matrices (see Section \ref{SparseRM}), down to interval length  $\frac{CK^2\log n}{np}$. 
We specialize our results to sparse SBM with finite many blocks, where a limiting law exists. Finally, we extend these results to an unbounded
number of blocks for the SBM, under certain conditions (see Section
\ref{SparseSBM}).

It should be said that our local laws for sparse general Wigner-type matrices
are not sharp enough to yield universality, unless $p$ is $\omega(1/n^{\varepsilon})$ for any $\varepsilon>0$. This is an artifact of the
use of the methods of \cite{tao2011random} and is also observable in \cite{tran2013sparse}. It is to be expected that they can be refined (by us or by
other researchers) in the near future to a point where universality
can be deduced.

\section{General Wigner-type Matrices}\label{GWM}
Let $M_n:=(\xi_{ij})_{1\leq i,j\leq n}$ be a random Hermitian matrix with variance profile $S_n=(s_{ij})_{1\leq i,j\leq n}$ such that $\xi_{ij}, 1\leq i\leq j\leq n $ are independent with
\begin{align}
\mathbb E\xi_{ij}=0, \mathbb E|\xi_{ij}|^2=s_{ij},	\notag
\end{align}
and compactly supported almost surely, i.e. $|\xi_{ij}|\leq K$ for some $K=o\left(\sqrt{\frac{n}{\log n}}\right)$.

For the variance profile $S_n$, we assume 
\begin{align}
c\leq s_{ij}\leq 1	\notag
\end{align}
for some constant $c>0$. Note this is equivalent to $c\leq s_{ij}\leq C$ by scaling. Define $W_n:=\frac{M_n}{\sqrt n}$.
The Stieltjes transform of the empirical spectral distribution of $W_n$ is given by
\begin{align}
s_n(z):=\frac{1}{n}\textnormal{tr}(W_n-zI)^{-1}	.\notag
\end{align}

We will show that $s_n(z)$ can be approximated by the solution of the following quadratic vector equation studied in \cite{ajanki2015quadratic}:
\begin{align}
m_n(z)&=\frac{1}{n}\sum_{k=1}^n g_n^{(k)}(z)	,\label{gwigner}\\
-\frac{1}{g_n^{(k)}(z)}&=z+\frac{1}{n}\sum_{l=1}^n s_{kl}g_n^{(l)}(z),  \quad 1\leq k\leq n. \label{gwignerqve}
\end{align}

From Theorem 2.1 in \cite{ajanki2015quadratic}, equation \eqref{gwignerqve}  has a unique set of solutions  $g_n^{(k)}(z): \mathbb H\to \mathbb H, 1\leq k\leq n$, which  are analytic functions on the complex upper half plane $\mathbb H:=\{z\in \mathbb C: \textnormal{Im}(z)>0\}$. The  unique solution $m_n(z)$ in equation \eqref{gwigner} is the Stieltjes transform of a probability measure $\rho_n$ with $\textnormal{supp}(\rho_n)\subset[-2,2]$ such that
\begin{align}
	\rho_n(x):=\lim_{\eta\downarrow 0} \frac{1}{\pi}\textnormal{Im} (m_n(x+i\eta)).\notag
\end{align} We use the following definition for bulk intervals of $\rho_n$.
\begin{definition}
	An interval $I$ of a probability density function $\rho$ on $\mathbb R$ is a bulk interval if there exists some fixed $\varepsilon>0$ such that 
	$\rho(x)\geq \varepsilon$, for any $x\in I$.
\end{definition}
	
We obtain the following local law of $M_n$ in the bulk.	

\begin{theorem}\label{locallawgwig}
\textnormal{(Local law  in the bulk)}
Let $M_n$ be a general Wigner-type matrix and $\rho_n$ be the probability measure corresponding to equations \eqref{gwigner},\eqref{gwignerqve}.
	 For any constant $\delta,C_1$, there exists a constant $C_2>0$ such that  with probability at least $1-n^{-C_1}$, the following holds.   For any  bulk interval $I$ of length $|I|\geq  \frac{C_2K^2\log n}{n}$ ,   the number of eigenvalues $N_I$ of $W_n$ in $I$ obeys the concentration estimate
	\begin{align}\label{result}
	 \left |N_I-n\int_I \rho_n(x)dx \right|\leq \delta n |I|.	\end{align} 	\end{theorem}

As a consequence, we obtained an optimal upper bound for eigenvectors that corresponds to eigenvalues of $W_n$ in the bulk interval.

\begin{theorem}\textnormal{(Optimal delocalization of eigenvectors in the bulk)}\label{deloc}
Let $M_n$ be a general Wigner-type matrix. For any constant $C_1>0$ and  any bulk interval $I$  such that eigenvalue $\lambda_i(W_n)\in I$, with probability at least $1-n^{-C_1}$, there is a constant $C_2$ such that the corresponding unit eigenvector $u_i(W_n)$ satisfies 
\begin{align}\|u_i(W_n)\|_{\infty}\leq \frac{C_2K\log^{1/2} n}{\sqrt n}.	\notag
\end{align}
\end{theorem}

\begin{remark}
Theorem \ref{locallawgwig} and Theorem \ref{deloc} also hold for general Wigner-type matrix whose entries $\xi_{ij}$'s  are sub-gaussian	 with sub-gaussian norm bounded by $K$. As mentioned in Remark 4.2 in \cite{o2016eigenvectors},  the proof follows in the same way by using the inequality in Theorem 2.1 in \cite{rudelson2013hanson} for sub-gaussian concentration instead of Lemma 1.2 in \cite{vu2015random} for $K$-bounded entries.
\end{remark}

We use standard methods from \cite{tao2011random}, adapted to fit the model considered here.

\subsection{Proof of Main Results}
\subsubsection{Proof of Theorem \ref{locallawgwig}}
For any $0<\varepsilon<\frac{1}{2}$, and constant $C_1>0$, define a region
\begin{align}\label{DefofC3}
D_{n,\varepsilon}:=\{z\in \mathbb C: \rho_n(\textnormal{Re}(z))\geq \varepsilon, \textnormal{Im}(z)\geq \frac{C_3^2K^2\log n}{n\delta^6}\}	
\end{align}
for some constant $C_3>0$ to be decided later.

 Let  $W_{n,k}$ be  the matrix $W_n$ with the $k$-th row and column removed, and $a_k$ be the $k$-th row of $W_n$ with the $k$-th element removed. 
 
 Let $(W_{n}-zI)^{-1}:=(q_{ij}^{(n)})_{1\leq i,j\leq n}$. From Schur's complement lemma (Theorem A.4 in \cite{bai2010spectral}), we have 
\begin{align}
q_{kk}=	\frac{1}{-\frac{\xi_{kk}}{\sqrt{n}} -z-Y_k},\notag
\end{align}
where 
\begin{align}
Y_k=a_k^*(W_{n,k}-zI)^{-1}a_k.	\notag
\end{align}
Let 
\begin{align}\label{fqve}
f_n^{(k)}(z):=\frac{1}{-\frac{\xi_{kk}}{\sqrt{n}} -z-Y_k},
\end{align}
then we can write $s_n(z)$ as,
 \begin{align}\label{strans}
s_n(z)=\frac{1}{n}\textnormal{tr}(W_n-zI)^{-1}=\frac{1}{n}\sum_{k=1}^n f_n^{(k)}(z).
            \end{align}
 We first estimate $Y_k$ to derive a perturbed version of \eqref{gwignerqve}.  Let $(W_{n,k}-zI)^{-1}:=(q_{ij}^{(n,k)})_{1\leq i,j\leq n-1}$, 
and $S_n^{(k)}$ be a diagonal matrix whose diagonal elements are the $k$-th row of $S_n$ with the $k$-th entry removed. We have
\begin{align}\label{expectedY}
\mathbb E[Y_k|W_{n,k}]&=\mathbb E[a_k^*(W_{n,k}-zI)^{-1}a_k|W_{n,k}]\notag\\
                      &=\sum_{i=1}^{n-1} q_{ii}^{(n,k)}\mathbb E|a_{ki}|^2 \notag\\
                      &=\sum_{i=1}^{n-1} q_{ii}^{(n,k)}s_{ki}                     \notag\\
             &=\frac{1}{n}\textnormal{tr}[(W_{n,k}-zI)^{-1}S_n^{(k)}].
 \end{align}
 The following 2 lemmas give estimates for $Y_k$, and the proofs are deferred for the next section.
 \begin{lemma}\label{girkoq} Let $\Sigma_n^{(k)}$ be the diagonal matrix whose diagonal elements are the $k$-th row of $S_n$. For any $k, 1\leq k\leq n$, and any fixed $z$ with $\textnormal{Im}(z)\geq\frac{K^2C_3^2\log n}{n\delta^6}$,
	\begin{align} \mathbb E[Y_k|W_{n,k}]=\frac{1}{n}\textnormal{tr}[(W_{n}-zI)^{-1}\Sigma_n^{(k)}]+O\left(\frac{1}{n\eta}\right),
	\notag
	\end{align}where the constant in the $O\left(\frac{1}{n\eta}\right)$ term is independent of $z$.

\end{lemma}
A similar estimate holds for $Y_k$ itself.
\begin{lemma} \label{weightedprojectionq}For any constant $C>0$, one can choose the constant $C_3$ defined in \eqref{DefofC3} sufficiently large   such that for any $k,1\leq k\leq n$, $z\in D_{n,\varepsilon}$,  one has
	\begin{align} \label{traceyk1}
	 Y_k-\frac{1}{n}\textnormal{tr}[(W_{n}-zI)^{-1}\Sigma_n^{(k)}]= o(\delta^2) 	
	\end{align}
	 with probability at least $1-n^{-C-10}$. 
\end{lemma}

 With the help of Lemmas \ref{girkoq} and \ref{weightedprojectionq}, note that, since $\frac{|\xi_{kk}|}{\sqrt{n}}=o(\delta^2)$, 
\begin{align}\label{weightedtrace1}
\frac{1}{n}\textnormal{tr}[(W_{n}-zI)^{-1}\Sigma_n^{(k)}]=\frac{1}{n}\sum_{l=1}^{n}s_{kl}f_n^{(l)},
\end{align}
 and combining \eqref{fqve}, \eqref{traceyk1}, \eqref{weightedtrace1}, we have 
\begin{align}\label{stochasticqve}
f_n^{(k)}(z)+\frac{1}{\frac{1}{n}\sum_{l=1}^{n}s_{kl}f_{n}^{(l)}(z)+z+o(\delta^2)}=0,\quad 1\leq k\leq n
\end{align}

with probability at least $1-n^{-C-9}$.

The next step involves using the stability analysis of quadratic vector equations provided in \cite{ajanki2015quadratic} to compare the solutions to \eqref{stochasticqve} and \eqref{gwignerqve}.  We have the following estimate.
\begin{lemma}\label{qvelemma1} For any constant $C>0$, one can choose $C_3$ in the \eqref{DefofC3} sufficiently large such that
\begin{align}\sup_{1\leq k\leq n}|f_n^{(k)}(z)-g_n^{(k)}(z)|=o(\delta^2), 
\end{align}

 for all $z\in  D_{n,\varepsilon}$ uniformly with probability at least $1-n^{-C-2}$.\end{lemma}

With Lemma \ref{qvelemma1}, we have for any $C>0$, there exists $C_3>0$  in \eqref{DefofC3} such that \begin{align}\label{smn}
|s_n(z)-m_n(z)|=\left|\frac{1}{n}\sum_{k=1}^n f_n^{(k)}(z)-\frac{1}{n}\sum_{k=1}^n g_n^{(k)}(z)\right|=o(\delta^2)\end{align}
uniformly for all $z\in D_{n,\varepsilon}$ with probability at least $1-n^{-C}$.

To complete the proof of Theorem \ref{locallawgwig}, we need the following well-known connection between the Stieltjes transform and empirical spectral distribution, shown for example in Lemma 64 in \cite{tao2011random} and also Lemma 4.1 in \cite{vu2015random}.
\begin{lemma} \label{control1}
 Let $M_n$ be a general Wigner-type matrix. Let $\varepsilon,\delta>0$. for any constant $C_1>0$, there exists a constant $C>0$ such that  suppose that one has the bound 
\begin{align}|s_n(z)-m_n(z)|\leq \delta	\notag
\end{align}
 with probability at least $1-n^{-C}$ uniformly for all $z\in D_{n,\varepsilon}$, then for any bulk interval $I$ with $|I|\geq \max\{2\eta,\frac{\eta}{\delta}\log\frac{1}{\delta}\}$ where $\eta=\frac{C_3^2K^2\log n}{n}$, one has 
\begin{align} \left| N_I-n\int_{I}\rho_n(x)dx\right|\leq\delta n|I| 	\notag
\end{align}
 with  probability at least $1-n^{-C_1}$.
\end{lemma}

From \eqref{smn}, for any constant $C_1>0$, we can choose $C_3$ in \eqref{DefofC3} large enough such that $$|s_n(z)-m_n(z)|\leq \delta$$ uniformly for all $z\in D_{n,\varepsilon}$ with probability $1-n^{-C}$, where $C$ is the constant in the assumption of Lemma \ref{control1}, then Theorem \ref{locallawgwig} follows from Lemma \ref{control1}.
%By taking $C_3>\max\{C, C_1+10^4\}$ in \eqref{DefofC3}, Theorem \ref{locallawgwig} then follows from Lemma \ref{control1}.

\subsubsection{Proof of Theorem \ref{deloc}}

The proof is based on Lemma 41 from \cite{tao2011random} given below. 
\begin{lemma}\label{lemm29}
Let $W_n$ be a $n\times n$ Hermitian matrix, and $W_{n,k}$ be the submatrix of $W_n$ with $k$-th row and column removed, and let $u_i(W_n)$ be a unit eigenvector of $W_n$ corresponding to $\lambda_i(W_n)$, and $x_k$ be the $k$-th coordinate of $u_i(W_n)$. Suppose that none of the eigenvalues of $W_{n,k}$ are equal to $\lambda_i(W_n)$. Let $a_k$ be the $k$-th row of $W_n$ with  of $k$-th entry removed; then
\begin{align}\label{entry}
|x_k|^2=\frac{1}{1+\sum_{j=1}^{n-1}(\lambda_j(W_{n,k})-\lambda_i(W_n))^{-2}|u_j(W_{n,k})^*a_k|^2},
\end{align}
 where $u_j(W_{n,k})$ is a unit eigenvector corresponding to $\lambda_j(W_{n,k})$.
\end{lemma}

Another lemma we need is a weighted projection lemma for random vectors with different variances. It is a slight generalization of  Lemma 1.2 in \cite{vu2015random}. Note that in the below \begin{align}\mathbb E|u_j^*X|^2=\textnormal{tr}(u_ju_j^*\Sigma),\notag
\end{align} 	 and the proof follows verbatim as in \cite{vu2015random}.
\begin{lemma}\label{same}
	Let $X=(\xi_1,\dots,\xi_n)$ be a $K$-bounded random vector in $\mathbb C^n$ such that $\textnormal{Var}(\xi_i)=\sigma_i^2$, $0\leq \sigma_i^2\leq 1$. Then there are constants $C,C'>0$ such that the following holds. Let $H$ be a subspace of dimension $d$ with an orthonormal basis $\{u_1,\dots,u_d\}$, and $\Sigma=\textnormal{diag}(\sigma_1^2,\dots,\sigma_n^2)$. Then for any $1\geq r_1,\dots, r_d\geq 0$,
	\begin{align}\label{reweight}
	\mathbb P\left(\left|\enskip \sqrt{\sum_{j=1}^d r_j|u_j^*X|^2}-\sqrt{\sum_{j=1}^d r_j \textnormal{tr}(u_ju_j^*\Sigma)}\enskip\right|\geq t\right)\leq C\exp (-C'\frac{t^2}{K^2}).
	\end{align} In particular, by squaring, it follows that
	\begin{align} \label{252}
	\mathbb P\left(\left|\sum_{j=1}^d r_j|u_j^*X|^2-\sum_{j=1}^d r_j \textnormal{tr}(u_ju_j^*\Sigma)\right|\geq 2t\sqrt{\sum_{j=1}^d r_j \textnormal{tr}(u_ju_j^*\Sigma)}+t^2\right)\leq C\exp (-C'\frac{t^2}{K^2}).
	\end{align}
\end{lemma}

Below we show how delocalization follows from Lemma \ref{lemm29}, Lemma \ref{same} and Theorem \ref{locallawgwig}.
For any $C_1>0$ and any $\lambda_i(W_n)$ in the bulk, by Theorem \ref{locallawgwig}, one can find an  interval $I$ centered at $\lambda_i(W_n)$ and $|I|=\frac{K^2C_2\log n}{n}$ for some sufficiently large $C_2$ such that $N_I\geq \delta_1n |I|$ for some small $\delta_1>0$ with probability at least $1-n^{-C_1-3}$. By Cauchy interlacing law, we can find a set $J\subset \{1,\dots, n-1\}$ with $|J|\geq N_I/2$ such that $|\lambda_j(W_{n,k})-\lambda_i(W_n)|\leq |I|$ for all $j\in J$.  Let $X_k$ be the $k$-th column of $M_n$ with the $k$-th entry removed. Note that from Lemma \ref{same}, by taking $r_j=1, j\in J,$ and  $t=C_3K\sqrt{\log n}$ for some constant $C_3\geq \frac{C_1+3}{C'}$ in \eqref{reweight}, using assumption $s_{ij}\geq c$, we have
\begin{align}\label{it}
	\sqrt{\sum_{j\in J}|u_j(W_{n,k})^*X_k|^2}&\geq \sqrt{\sum_{j\in J} \textnormal{tr}(u_j(W_{n,k})u_j^*(W_{n,k})\Sigma)}-C_3K\sqrt{\log n}\notag\\
	&\geq \sqrt{c|J|}-C_3K\sqrt{\log n}\notag\\
	&\geq (\sqrt{c}-\frac{C_3}{\sqrt{C_2\delta_1/2}})\sqrt{|J|}
\end{align} with probability at least $1-n^{-C_1-3}$. By choosing $C_2$ sufficiently large, \eqref{it} implies
\begin{align*}
	\sum_{j\in J}|u_j(W_{n,k})^*X_k|^2\geq C'|J|
\end{align*}
for some constant $C'>0$ with probability at least $1-n^{-C_1-3}$.
 By \eqref{entry}, 
\begin{align*}
|x_k|^2&=\frac{1}{1+\sum_{j=1}^{n-1}(\lambda_j(W_{n,k})-\lambda_i(W_n))^{-2}|u_j(W_{n,k})^*\frac{X_k}{\sqrt n}|^2}	\\
&\leq \frac{1}{1+\sum_{j\in J}(\lambda_j(W_{n,k})-\lambda_i(W_n))^{-2}|u_j(W_{n,k})^*\frac{X_k}{\sqrt n}|^2}   \\
&\leq \frac{1}{1+n^{-1}|I|^{-2}\sum_{j\in J}|u_j(W_{n,k})^*X_k|^2} \\
&\leq \frac{1}{1+n^{-1}|I|^{-2}C'|J|} \\
& \leq \frac{2|I|}{C'\delta_1}\leq \frac{K^2C_4^2\log n}{n}
\end{align*}
for some constant $C_4$ with probability at least $1-2n^{-C_1-3}$. Thus by taking a union bound, $\|u_i\|_{\infty}\leq \frac{C_4K\sqrt {\log n}}{\sqrt n}$ with probability at least $1-n^{-C_1}$ for all $1\leq i\leq n$.

\subsection{Proof of Auxiliary Lemmas}
We now prove the all the lemmas in the proof of Theorem \ref{locallawgwig}.

\subsubsection{Proof of Lemma \ref{girkoq}}

Let $\displaystyle \eta:=\frac{C_3^2K^2\log n}{n\delta^6}
$
and $z:=x+\sqrt{-1}\cdot  \eta$.
By \eqref{expectedY},  it suffices to show for all $1\leq k\leq n$,
\begin{align}\label{toproof}
\left|\textnormal{tr}[(W_n-zI)^{-1}\Sigma_n^{(k)}]-\textnormal{tr}[(W_{n,k}-zI)^{-1}S_n^{(k)}]\right|\leq \frac{1}{\eta}.
\end{align}

We will use the following result known as  Lemma 1.1 in Chapter 1 of \cite{girko2001theory}.

 \begin{lemma}\label{girkolemma1}
 Let $\vec{c}=(c_1,\dots,c_n)$ be a real column vector, and $M_n=(\xi_{ij})_{n\times n}$ be a Hermitian matrix, for any $z$ with and $\textnormal{Im} z>0$,   we have, for any $1\leq k\leq n$, 
\begin{align}
 \vec{c}^T(M_n-zI)^{-1}\vec{c}-\vec{c_k}^T(M_{n,k}-zI)^{-1}\vec{c_k}=\frac{c_k^2-\vec{\xi_k}^*R_k(2c_k\vec{c_k})+\vec{\xi_k}^*R_k\vec{c_k}\vec{c_k}^TR_k\vec{\xi_k}}{\xi_{kk}-z-\vec{\xi_k}^* R_k\vec{\xi_k}} \notag
 \end{align}
  where $R_k=(M_{n,k}-zI)^{-1}$, $\vec{c_k}$ is the vector $\vec{c}$ with  the $k$-th coordinate removed, and $\vec{\xi_k}$ is the $k$-th column of $M_n$ with the $k$-th element removed.
\end{lemma}

 We introduce a real random vector $\vec{c}=(c_1,\dots,c_n)$ whose coordinates are mean zero, independent variables also independent of $W_n$ with $\textnormal{Var}(c_i)=s_{ki}$ for $ 1\leq i\leq n$. 
 
 Apply Lemma \ref{girkolemma1} to $W_n$ and $\vec{c}$. We have $R_k=(W_{n,k}-zI)^{-1}$, and
 
 \begin{align}
 	\vec{c}^T(W_n-zI)^{-1}\vec{c}-\vec{c_k}^T(W_{n,k}-zI)^{-1}\vec{c_k}=\frac{c_k^2-a_k^*R_k(2c_k\vec{c_k})+a_k^*R_k\vec{c_k}\vec{c_k}^TR_ka_k}{\frac{\xi_{kk}}{\sqrt n}-z-Y_k}.\notag
 \end{align}
By taking the conditional expectation with respect to $c$, conditioned on $W_{n}$, we have 
\begin{align}
	\mathbb E[\vec{c}^T(W_n-zI)^{-1}\vec{c}-\vec{c_k}^T(W_{n,k}-zI)^{-1}\vec{c_k}\mid W_n]&=\frac{s_{kk}+a_k^* R_kS_n^{(k)}R_k a_k}{\frac{\xi_{kk}}{\sqrt {n }}-z-Y_k}	.\notag
	\end{align}
	Calculating the left hand side yields
	\begin{align}
	\text{tr}[(W_n-zI)^{-1}\Sigma_n^{(k)}]-\text{tr}[(W_{n,k}-zI)^{-1}S_n^{(k)}]&=\frac{s_{kk}+a_k^*R_kS_n^{(k)}R_k a_k}{\frac{\xi_{kk}}{\sqrt{ n }}-z-Y_k}.\notag
\end{align}
 Since we have
\begin{align*}
	|s_{kk}+a_k^* R_k S_n^{(k)}R_k a_k|&\leq 1+| a_k^* R_kS_n^{(k)}R_k a_k|\\
	&\leq   1+a_k^*((W_{n,k}-xI)^2+\eta^2I)^{-1}a_k,
\end{align*}
and
\begin{align}\textnormal{Im}\left(\frac{\xi_{kk}}{\sqrt n}-z-Y_k\right)=-\eta\left(1+
a_k^*((W_{n,k}-xI)^2+\eta^2I)^{-1} a_k\right),\notag
\end{align}

 \eqref{toproof} holds. 	This completes the proof of Lemma \ref{girkoq}.

\subsubsection{Proof of Lemma \ref{weightedprojectionq}}

%Taking  $t=\frac{cK}{20}\log n, c_j=1, 1\leq j\leq d$ we have with overwhelming probability,
%
%
%\begin{equation}\label{eqn:einstein}
%\|P_H(X)\|^2\geq \sum_{j=1}^d \textnormal{tr}(u_ju_j^* \Sigma)- \frac{cK}{10}\log n\sqrt{\sum_{j=1}^d \textnormal{tr}(u_ju_j^* \Sigma)}-\frac{K^2c^2}{400}\log ^2 n	
%\end{equation}

We need a preliminary bound on the number of eigenvalues in a short interval. The following Lemma is similar to Proposition 66 in \cite{tao2011random}.

\begin{lemma}\label{prelim}
	For any constant $C_1>0$, there exists a constant  $C_2>0$ such that for any interval $I\subset \mathbb R$ with $|I|\geq \frac{C_2K^2\log n}{n}$, one has
	\begin{align}
	N_I(W_n)=O( n|I|)	\label{roughestimate}
	\end{align}
 with probability at least $1-n^{-C_1}$.\end{lemma}
 
 \begin{proof}
 By the union bound, it suffices to show that the failure probability for \eqref{roughestimate} is less than $1-n^{-C_1-1}$ for
 $$|I|=\eta:=\frac{C_2K^2\log n}{n}
 $$	for some sufficiently large $C_2$. By 
 \begin{align}\label{233}
	\textnormal{Im}(s_n(x+\sqrt{-1}\eta))=\frac{1}{n}\sum_{i=1}^n\frac{\eta}{\eta^2+(\lambda_i(W_n)-x)^2},
\end{align}  it suffices to show that the event
 \begin{align}
 N_I\geq Cn\eta	\label{254}
 \end{align}
 and 
  \begin{align} \label{255}
  \textnormal{Im}(s_n(x+\eta\sqrt{-1}))\geq C	
 \end{align}
fails with probability at least $1-n^{-C_1-1}$ for some large absolute constant $C>1$. Suppose we have \eqref{254}, \eqref{255}, by \eqref{233},
\begin{align}
\frac{1}{n}\sum_{k=1}^n\left|\textnormal{Im}\left( \frac{1}{\frac{\xi_{kk}}{\sqrt n} -(x+\eta\sqrt{-1})-Y_k}\right)\right|	\geq C. \notag
\end{align}
Using the bound $\displaystyle \left|\text{Im} \left(\frac{1}{z} \right)\right|\leq \frac{1}{|\text{Im}(z)|}$, it implies

\begin{align}\label{257}
\frac{1}{n}\sum_{k=1}^n\frac{1}{|\eta+\textnormal{Im}(Y_k)|}\geq C.	
\end{align}
 
 Note that $$W_{n,k}=\sum_{j=1}^{n-1}\lambda_j(W_{n,k})u_j^*(W_{n,k})u_j(W_{n,k}),$$
where $u_j(W_{n,k}), 1\leq j\leq n-1$ are orthonormal basis of $W_{n,k}$, 
one has 

\begin{align}
Y_k&=a_k^*(W_{n,k}-zI)^{-1}a_k=\sum_{j=1}^{n-1}\frac{|u_j^*(W_{n,k})a_k |^2}{\lambda_j(W_{n,k})-(x+\eta \sqrt{-1})}\notag
\end{align}
and hence
\begin{align}
\textnormal{Im} Y_k\geq \eta \sum_{j=1}^{n-1}\frac{|u_j^*(W_{n,k})a_k |^2}{\eta^2+(\lambda_j(W_{n,k}-x)^2}.	\notag
\end{align}

On the other hand, from \eqref{254}, by Cauchy interlacing theorem, we can find an index set $J$ with $|J|\geq \eta n$ such that $\lambda_j(W_{n,k})\in I$ for all $j\in J$, then we have 
\begin{align}\label{259}
\textnormal{Im}(Y_k)\geq \frac{1}{2\eta}\sum_{j\in J}|u_j^*(W_{n,k})a_k|^2=\frac{1}{2\eta}\|P_{H_k}a_k\|^2,	
\end{align}
where $P_{H_k}$ is the orthogonal projection onto a subspace $H_k$ spanned by eigenvectors $u_j(W_{n,k}),j\in J$. From \eqref{257}, \eqref{259}, we have
\begin{align}\label{2510}
	\frac{1}{n}\sum_{k=1}^n \frac{2\eta}{2\eta^2+\|P_{H_k}a_k\|^2}\geq C.
\end{align}
On the other hand, taking $r_j=1, 1\leq j\leq d, d=|J|$ and $t=C_4 K\sqrt{\log n}$ for some suffciently large $C_4$ in \eqref{252}, using assumption $s_{ij}\geq c$, we have that $\|P_{H_k}(a_k)\|^2=\Omega(\eta)$ with probability at least $1-O(n^{-C'C_4})\geq 1-n^{-C_1-5}$. Taking the union bound over all possible choice of $J$, we have \eqref{2510} holds with probability at least $1-n^{-C_1-1}$. The claim then follows by taking $C$ sufficiently large.

Now we are ready to prove Lemma \ref{weightedprojectionq}. 
	From Lemma \ref{girkoq}, it suffices to show
	\begin{align} Y_k-\mathbb E[Y_k|W_{n,k}]=o (\delta^2), \quad 1\leq k\leq n,		
	\end{align}
	with probability at least $1-n^{-C-10}$. We can write 
	\begin{align}Y_k=\sum_{j=1}^{n-1}\frac{|u_j^*(W_{n,k})a_k|^2}{\lambda_j(W_{n,k})-z},	\notag
	\end{align}
where $\{u_j(W_{n,k})\}_{j=1}^{n-1}$ are orthonormal eigenvectors of $W_{n,k}$.
	Moreover,
	\begin{align*}
		\mathbb E[Y_k|W_{n,k}]&=\frac{1}{n}\textnormal{tr}[(W_{n,k}-zI)^{-1}S_n^{(k)}] \\
		                      &=\frac{1}{n}\textnormal{tr}\left[\sum_{j=1}^{n-1}\frac{1}{\lambda_j(W_{n,k})-z}u_j(W_{n,k})u_j^*(W_{n,k})S_n^{(k)}\right] \\
		                      &=\frac{1}{n}\sum_{j=1}^{n-1}\frac{\textnormal{tr}[u_j(W_{n,k})u_j^*(W_{n,k})S_n^{(k)}]}{\lambda_j (W_{n,k})-z}.
	\end{align*}
	Let $X_k=\sqrt{n}a_k$, and define
	\begin{align}t_j:=
	 |u_j(W_{n,k})^*X_k|^2-\textnormal{tr}[u_j(W_{n,k})u_j^*(W_{n,k})S_n^{(k)}].\notag
	\end{align}
	It suffices to show that
	$$\left|Y_k-\mathbb E[Y_k|W_{n,k}]\right|=\frac{1}{n}\left|\sum_{j=1}^{n-1}\frac{t_j}{\lambda_j(W_{n,k})-x-\sqrt{-1}\eta}\right|=o(\delta^2)$$
	with probability at least $1-n^{-C-10}$.
	The remaining part of the proof goes through in the same way as in the proof of Lemma 5.2 in \cite{vu2015random}  with Lemma \ref{same} and Lemma \ref{prelim}. Then Lemma \ref{weightedprojectionq} follows.
\end{proof}
\subsubsection{Proof of Lemma \ref{qvelemma1}}
We define $g_n(z,x):=g_n^{(k)}(z)$ if $x\in [\frac{k-1}{n},\frac{k}{n}), 1\leq k\leq n$ and 
\begin{align}\label{varianceprofile}
	S_n(x,y):=s_{ij} \quad \text{ if } x\in[\frac{i-1}{n},\frac{i}{n}), y\in [\frac{j-1}{n},\frac{j}{n}).  
\end{align}

 Then \eqref{gwignerqve} can be written as
\begin{align}
m_n(z)&=\int_{0}^{1}g_n(z,x) dx,
\label{integralqve}\\
  -\frac{1}{g_n(z,x)}&=z+\int_{0}^{1}S_n(x,y) g_n(z,y) dy, \label{Q}
\end{align} for all $x\in[0,1]$. Similarly, define 
$f_n(z,x):=f_n^{(k)}(z)$ if $x\in [\frac{k-1}{n},\frac{k}{n}), 1\leq k\leq n$. Then we can write \eqref{strans} and \eqref{stochasticqve} as 
\begin{align*}
s_n(z)&=\int_{0}^1 f_n(z,x) dx\\
	-\frac{1}{f_n(z,x)}&=z+\int_{0}^{1}S_n(x,y)f_n(z,y)dy+d_n(z,x),                     
\end{align*}
where, for any fixed $z$ from \eqref{stochasticqve}, \begin{align}\label{infi}
\|d_n(z)\|_{\infty}:=\sup_{x\in[0,1]} |d_n(z,x)|=o(\delta^2)	
 \end{align}
 with probability at least least $1-n^{-C-9}$ for any fixed $z\in D_{n,\varepsilon}$.

The following lemma follows from Theorem 2.12 in \cite{ajanki2015quadratic} which controls the stability of equation \eqref{Q} in the bulk. Here we use the fact that $c\leq s_{ij} \leq 1 $ to guarantee the assumptions of $S_n$ in Theorem 2.12 in \cite{ajanki2015quadratic}. Define
\begin{align}\Lambda(z):=\sup_{x\in [0,1]}|f_n(z,x)-g_n(z,x)|.\notag
\end{align}
\begin{lemma} \label{QVEstab}
 For any fixed $z\in D_{n,\varepsilon}$,  there  exist  constants $\lambda, C_5>0$ depending on $\varepsilon$ but independent of $n$  such that  for $z\in D_{n,\varepsilon}$,
\begin{align}\Lambda(z)\mathbf{1}\{\Lambda(z)\leq \lambda \}  \leq C_5\|d_n (z)\|_{\infty}.\label{qvestabilitybound}
\end{align}
\end{lemma}
\begin{proof}
Since the variance satisfies $c\leq s_{ij}\leq 1$, $S_n$ satisfies condition A1-A3 in Chapter 1 of \cite{ajanki2015quadratic}. Especially, it implies condition A3 with $L=1$. 

From the lower bound  $\rho_n(z)\geq \varepsilon$, Lemma 5.4 (i) in \cite{ajanki2015quadratic} implies,
\begin{align}
\sup_{1\leq k\leq n}|g_n^{(k)}(z)|\leq \frac{1}{\varepsilon}<\infty 	 \notag
\end{align}
for any $z$ with $\textnormal{Re}(z)\in I$, $\textnormal{Im}(z)>0$. Then the assumptions in Theorem 2.12 in \cite{ajanki2015quadratic} holds. \eqref{qvestabilitybound} then follows from Theorem 2.12 in \cite{ajanki2015quadratic}. \end{proof}

\begin{remark}
	Lemma \ref{QVEstab} is a stability result for the solution of \eqref{Q}, which is deterministic and does not require moment assumptions on the random matrix $M_n$.
\end{remark}

From Lemma \ref{QVEstab} and \eqref{infi},
 we have for any fixed $z\in D_{n,\varepsilon}$,
 \begin{align}\label{pointwise}\Lambda(z)\mathbf{1}\{\Lambda(z)\leq \lambda  \}  =o(\delta^2) 
\end{align} with probability at least $1-n^{-C-9}$.

We proceed with a continuity argument as in the proof of Theorem 3.2 in the bulk (Section 3.1 in \cite{ajanki2015quadratic}) to show \eqref{pointwise} holds uniformly for $z\in D_{n,\varepsilon}$ with  probability at least $1-n^{-C-2}$.

Now for any $0<\varepsilon'<\frac{\lambda}{4}$, we consider a line segment $$L=x+\sqrt{-1}\left[\frac{K^2C_3^2\log n}{n\delta^6},n\right]$$ for some fixed $x$ with $\rho_n(x)\geq \varepsilon, 0<\varepsilon<1/2$, and let $n$ be large enough such that 
$\frac{1}{n}<\varepsilon'$ 
and $\|d_n(z)\|_{\infty} \leq \varepsilon'$. let $L_n$ consist of $n^4$ evenly spaced points on $L$. Then we have  
\begin{align}\label{258}
\Lambda(z)\mathbf{1}\{\Lambda(z)\leq \lambda\}  \leq \varepsilon' 
\end{align}

for all $z\in L_n$ with probability at least $1-n^{-C-5}$.

 From Theorem 2.1 in \cite{ajanki2015quadratic}, $g_n(z,x)$ is the Stieltjes transform of a probability measure, hence the derivative of $g_n(z,x)$ is uniformly bounded by $\frac{1}{|\textnormal{Im}(z)|^{2}}\leq n^2$ for $z\in D_{n,\varepsilon}$. Similarly, for $f_n(z,x)$, from \eqref{fqve}, for $1\leq k\leq n$, 
\begin{align*}  
\left|\frac{\partial f_n^{(k)}(z)}{\partial z}\right|&=\left|\frac{1+\frac{\partial{Y_k}}{\partial z}}{(\frac{\xi_{kk}}{\sqrt{n}}-z-Y_k)^2}\right|\\
&\leq \left|\frac{1+a_k^*(W_{n,k}-zI)^{-2}a_k}{\frac{\xi_{kk}}{\sqrt{n}}-z-Y_k}\right|\frac{1}{|\frac{\xi_{kk}}{\sqrt{n}}-z-Y_k|}.
\end{align*}
 By theorem A.6. in \cite{bai2010spectral}, for $z=x+\sqrt{-1}\eta$,
 \begin{align}\left|\frac{1+a_k^*(W_{n,k}-zI)^{-2}a_k}{(\frac{\xi_{kk}}{\sqrt{n}}-z-Y_k)}\right|\leq \frac{1}{\eta},\notag
 \end{align}
  and 
  \begin{align}\left|\frac{\xi_{kk}}{\sqrt{n}}-z-Y_k\right|\geq \left|\textnormal{Im}(\frac{\xi_{kk}}{\sqrt{n}}-z-Y_k)\right|=\eta(1+a_k^*((W_{n,k}-xI)^2+\eta^2I)^{-1}a_k)\geq \eta,  	\notag
  \end{align}
 Note that for $z\in D_{n,\varepsilon}$, $\eta\geq\frac{K^2C_3^2\log n}{n\delta^6}\geq \frac{1}{n}$, we get 
 \begin{align}\left|\frac{\partial f_n^{(k)}(z)}{\partial z}\right|\leq \frac{1}{\eta^2}\leq n^2, \quad 1\leq k\leq n.\notag
 \end{align}
   So both $f_n(z,x)$ and $g_n(z,x)$ are $n^2$-Lipschitz function in $z$ for $z\in D_{n,\varepsilon}$. It follows that 
\begin{align}|\Lambda(z')-\Lambda(z)|\leq 2n^2|z'-z|,	\notag
\end{align}
 for any $z,z'\in L$. 
 We first claim that 
\begin{align}\label{614}
\Lambda(z)\mathbf{1}\{\Lambda(z)\leq \frac{\lambda \varepsilon}{2}\}\leq 2\varepsilon',
\end{align} for all $z\in L$ with probability at least $1-n^{-C-5}$.

Since $0<\varepsilon<1/2$, if $z\in L_n$, \eqref{614} is true from \eqref{258}. If $z\in L\setminus L_n$, choose some $z'\in L_n$ such that $|z-z'|\leq n^{-3}$. Suppose  $\displaystyle \Lambda(z)\leq \frac{\lambda\varepsilon}{2}$, note that
\begin{align}\label{25555}
|\Lambda(z')-\Lambda(z)|\leq 2n^2|z-z'|\leq \frac{2}{n}	,
\end{align}
which implies
\begin{align}\Lambda(z')\leq  \Lambda(z)+\frac{2}{n}\leq \frac{\lambda\varepsilon}{2}+\frac{2}{n}\leq \lambda	\notag
\end{align} with probability at least $1-n^{-C-5}$.
From  \eqref{258}, $\Lambda(z')\leq \varepsilon'$ with probability at least $1-n^{-C-5}$. From \eqref{25555},
\begin{align}\Lambda(z)\leq \Lambda(z')+\frac{2}{n}<2\varepsilon',
\end{align} with probability at least $1-n^{-C-5}$, 
therefore  \eqref{614} holds.

In the next step we show that the indicator function in \eqref{614} is identically equal to $1$.
From \eqref{258} we have 
$\displaystyle \Lambda(z)\not\in \left(2\varepsilon',\lambda/2\right) 
$
 with probability at least $1-n^{-C-5}$. 

Let $E$ be the event that
$\Lambda(z)\mathbf{1}\{\Lambda(z)\leq \frac{\lambda }{2}\}\leq 2\varepsilon'
$
happens. Conditioning on $E$, since $\Lambda(z)$ is $2n^2$-Lipschitz in $z$, and $L$ is simply connected, we have 
\begin{align}\Lambda(L):=\{\Lambda(z): z\in L\}\notag\end{align} is simply connected.
Therefore $\Lambda(L)$ is contained either in $[0,2\varepsilon']$ or $[\frac{\lambda}{2},\infty)$.

From \eqref{fqve}  we have for $1\leq k\leq n$,
 	\begin{align}
|f_n^{(k)}(z)|=\frac{1}{\left|-\frac{\xi_{kk}}{\sqrt{n}} -z-Y_k\right|}\leq \frac{1}{|\textnormal{Im}(z)|},\notag
\end{align}
 and  since $g_n^{(k)}(z)$ is a Stieltjes transform of a probability measure,  for $1\leq k\leq n$,
 \begin{align}
 	|g_n^{(k)}(z)|\leq \frac{1}{|\textnormal{Im}(z)|},\notag
 \end{align}
which implies 
\begin{align}
\Lambda(z)\leq \frac{2}{|\textnormal{Im}(z)|}.	\notag
\end{align}
Consider the point $z_n:=x+\sqrt{-1}\cdot n\in L$, we have
\begin{align}\Lambda(z_n)\leq \frac{2}{\textnormal{Im} (z_n)}=\frac{2}{n}\leq 2\varepsilon', \notag
\end{align} 
which implies $\Lambda(z_n)\in [0,2\varepsilon']$.
Hence for all $z\in L$,
$\Lambda(z)\leq 2\varepsilon'$ with probability at least $1-n^{-C-5}$, and the indicator function in \eqref{614} is identically equal to $1$.

	Now we extend the estimate to all $z\in D_{n,\varepsilon}$. Consider $n^3$ lines segments $$x_k+\sqrt{-1}\left[\frac{K^2C_3^2\log n}{n\delta^6},n\right], \rho_n(x_k)\geq \varepsilon, 1\leq k\leq n^3$$ such that the $n^2$-neighborhoods of points $\{x_k, 1\leq k\leq n^3\}$ cover any bulk interval of $\rho_n$.
	By the $2n^2$-Lipschitz property of $\Lambda(z)$ again, we can show 
$\Lambda(z)\leq 4\varepsilon'
$ for all $z$ with $\rho_n(\textnormal{Re}(z))>\varepsilon$, $\frac{K^2C_3^2\log n}{n\delta^6}\leq \textnormal{Im} z\leq n$, with probability at least $1-n^{-C-2}$.

 On the other hand, for all $z$ with $\textnormal{Im}(z)>n$, 
 \begin{align}
 \left \|f_n(z)-g_n(z)\right\|_{\infty}&\leq\frac{2}{\textnormal{Im} z}=O\left(  \frac{1}{n}\right).
 \end{align}

Combining these two cases, for all $z\in D_{n,\varepsilon}$  with probability at least $1-n^{-C-2}$,
 \begin{align}\|f_n(z)-g_n(z)\|_{\infty}=o(\delta^2).	\notag
 \end{align}
  This completes the proof of Lemma \ref{qvelemma1}.
\section{Applications: Sparse Matrices}
\subsection{Sparse General Wigner-type Matrices}\label{SparseRM}

Let $M_n$ be a  sparse general Wigner-type matrix with independent entries
$M_{ij}=\delta_{ij}\xi_{ij}	$ for $1\leq i\leq j\leq n$. Here $\delta_{ij}$ are  i.i.d. Bernoulli random variables which take value $1$ with probability 
$\displaystyle p=\frac{g(n)\log n}{n}$, where $g(n)$ is any function for which $g(n)\to\infty$ as $n\to\infty$, and $\xi_{ij}$ are independent random variables such that 
 \begin{align}
 \mathbb E\xi_{ij}=0, \mathbb E|\xi_{ij}|^2=s_{ij}, 	c\leq s_{ij}\leq 1,\notag
 \end{align}
and  in addition, $|\xi_{ij}|\leq K$ almost surely for $K=o(\sqrt{g(n)})$.

We can regard this model as the sparsification of  a general Wigner-type matrix by uniform sampling. Similar models were considered in  \cite{wood2012universality,luh2018sparse}. 

Considering the empirical spectral distribution of $W_n:=\frac{M_n}{\sqrt n p}$, we  specify a  local law for this model.

\begin{cor}\label{spars}
	Let $M_n$ be a sparse general Wigner-type matrix, let $\rho_n$ be the probability measure corresponding to equations \eqref{gwigner},\eqref{gwignerqve}.
	 For any constants $\delta,C_1>0$, there exists a constant $C_2>0$ such that  with probability at least $1-n^{-C_1}$, the following holds.   For any  bulk interval $I$ of length $|I|\geq  \frac{C_2K^2\log n}{np}$ ,   the number of eigenvalues $N_I$ of $W_n:=\frac{M_n}{\sqrt{np}}$ in $I$ obeys the concentration estimate
	\begin{align}\label{result}
	 \left |N_I-n\int_I \rho_n(x)dx \right|\leq \delta n |I|.	\end{align}
\end{cor}

\begin{proof}
Define \begin{align}
 	H_n:=\frac{M_n}{\sqrt p}=(h_{ij})_{1\leq i,j\leq n}.	\notag
 \end{align}
 Then $\displaystyle
 \mathbb Eh_{ij}=0, \mathbb E|h_{ij}|^2=s_{ij},$ and $\displaystyle |h_{ij}|\leq \frac{K}{\sqrt p}=o\left(\sqrt{\frac{n}{\log n}}\right).
 $
 \eqref{result} follows as a corollary of Theorem \ref{locallawgwig} for $H_n$.
\end{proof}
The infinity norm of eigenvectors in the bulk can be estimated in a similar way.
\begin{cor} \label{delocspars}
Let $M_n$ be a sparse general Wigner-type matrix and $W_n=\frac{M_n}{\sqrt {np}}$. For any constant $C_1>0$ and  any bulk interval $I$  such that eigenvalue $\lambda_i(W_n)\in I$, with probability at least $1-n^{-C_1}$, there is a constant $C_2$ such that the corresponding unit eigenvector $u_i(W_n)$ satisfies 
\begin{align}\|u_i(W_n)\|_{\infty}\leq \frac{C_2K\log^{1/2} n}{\sqrt {np}}.	\notag
\end{align}	
\end{cor}

\subsection{Sparse Stochastic Block Models}\label{SparseSBM}
\subsubsection{Finite Number of Classes}
Our analysis of  sparse random matrices applies to the adjacency matrices of sparse stochastic block models. 

Consider the  adjacency matrix $A_n=(a_{ij})_{1\leq i,j\leq n}$ of an SBM graph, where $A_n$ is a random real symmetric block matrix with $d^2$ blocks. Recall that we partition all indices $[n]$ into $d$ sets,
\begin{align}[n] = V_1 \cup V_2
\cup \ldots \cup V_d
\end{align}
such that $|V_i|=N_i$.  We assume 
$a_{ii}=0, 1\leq i\leq n.$	
and $a_{ij}$, $i\not =j$ are Bernoulli random variables such that if $a_{ij}$ is in the $(k,l)$-th block, $a_{ij}=1$ with probability $p_{kl}$ and $a_{ij}=0$ with probability $1-p_{kl}$. 

Let $\sigma_{kl}^2:= p_{kl}(1-p_{kl})$. Define $p:=\displaystyle\max_{kl} p_{kl}$ and $\sigma^2=p(1-p)$. Assume
\begin{align}
	p=\frac{g(n)\log n}{n},\notag
\end{align} where $\displaystyle\sup_n p<1$ and $g(n)\to \infty$  as $n\to \infty$. We also assume that
\begin{align}
 \frac{N_i}{n}&=\alpha_i+o\left(\frac{1}{g(n)}\right),  \label{SBM3}  \\
 \frac{\sigma_{kl}^2}{\sigma^2}& = c_{kl}+o\left(\frac{1}{g(n)}\right), \label{SBM4}
\end{align}   where  $\alpha_i>0, 1\leq i\leq d$ and $c_{kl}\geq c>0, 1\leq k,l\leq d$ for some constant $c$. The quadratic vector equation becomes
\begin{align}\label{SBMq}
	m(z)&=\sum_{k=1}^d \alpha_k g_{k}(z)\\
-\frac{1}{g_k(z)}&=z+\sum_{l=1}^d\alpha_l c_{kl}g_l(z). \label{SBMQ}
\end{align}
We state the following local law for sparse SBM.

\begin{cor}\label{SSBM}
Let $A_n$ be the adjacency matrix of a stochastic block model with the assumptions above, let $\rho$ be the probability measure corresponding to equation \eqref{SBMq}.
	 For any constant $\delta,C_1>0$, there exists a constant $C_2>0$ such that  with probability at least $1-n^{-C_1}$, the following holds.   For any  bulk interval $I$ of length $|I|\geq  \frac{C_2\log n}{np}$,   the number of eigenvalues $N_I$ of $\frac{A_n}{\sqrt{n\sigma}}$ in $I$ obeys the concentration estimate
	\begin{align} 
	 \left |N_I-n\int_I \rho(x)dx \right|\leq \delta n |I|.	\notag\end{align} \end{cor}

	\begin{proof} 

We have the following well-known Cauchy Interlacing Lemma, appearing for example, as Lemma 36 from \cite{tao2011random}.
\begin{lemma}\label{cauchy}
Let $A,B$ be symmetric matrices with the same size and $B$ has rank $1$. Then for any interval $I$, we have \begin{align}|N_I(A+B)-N_I(B)|\leq 1.\end{align} 
where $N_I(M)$ is the number of eigenvalues of $M$ in $I$.
\end{lemma}
Let $\tilde{A}_n$ be the matrix whose off diagonal entries are equal to $A_n$ and 
\begin{align}\label{tilde}
\tilde{a}_{ii}=p_{kk} 	
\end{align} if $(i,i)$ is in the $k$-th block.

From Lemma \ref{cauchy}, since rank $\mathbb E(\tilde{A}_n)=d$, we have 
\begin{align}|N_I(A_n)-N_I(A_n-\mathbb E(\tilde{A	}_n))|\leq d=o(n|I|).\notag\end{align} 
Therefore it suffices to prove the local law for 
\begin{align}W_n=\frac{A_n-\mathbb E\tilde{A}_n}{\sqrt n\sigma}.\notag\end{align}
 Let $\frac{A_n-\mathbb E\tilde{A}_n}{\sigma}=(\xi_{ij})_{1\leq i,j\leq n}$.
	By Schur's complement, we can write the Stieltjes transform of the empirical measure $s_n(z)$ in the following way,
 \begin{align}
s_n(z)=\frac{1}{n}\sum_{k=1}^n\frac{1}{-\frac{\xi_{kk}}{\sqrt{n}} -z-Y_k}\notag
            \end{align}
We do the following partition of $s_n(z)$ into $d$ parts:
\begin{align}
	s_n(z):=\sum_{l=1}^d \frac{N_l}{n}f_n^{(l)}(z),\notag
\end{align}
where
\begin{align}\label{fln}
 f_n^{(l)}(z):=\frac{1}{N_l}\sum_{k\in V_l}\frac{1}{-\frac{\xi_{kk}}{\sqrt{n}}-z-Y_k}.
\end{align}
The $k$-th diagonal element in $\frac{A_n-\mathbb E\tilde{A_n}}{\sigma}$ is $\displaystyle\frac{-p_{kk}}{\sqrt{n}\sigma}=o(1).$
Similar with  \eqref{stochasticqve},  we have
\begin{align}\label{stochasticf}
-\frac{1}{f_n^{(l)}(z)}=\sum_{m=1}^{d}\frac{N_{m}}{n}c_{ml}f_{n}^{(m)}(z)+z+o(1), \quad 1\leq l\leq d
\end{align}
for any $z\in D_{n,\varepsilon}$
with probability at least $1-n^{-C-9}$.
Using the assumptions \eqref{SBM3}, \eqref{SBM4} and the fact that $|f_n^{(l)}|\leq \frac{1}{\eta}$,  we have
\begin{align}
-\frac{1}{f_n^{(l)}(z)}=z+\sum_{m=1}^{d}\alpha_m c_{ml}f_{n}^{(m)}(z)+o(1), \quad 1\leq l\leq d\label{perturb}
\end{align} for any fixed $z\in D_{n,\varepsilon} $ with probability at least $1-n^{-C-9}$.

Since $d$ is fixed and all coefficients $c_{kl}, 1\leq k,l\leq d$ in \eqref{SBMQ} are positive and bounded, from Theorem 2.10 in \cite{ajanki2015quadratic},
\begin{align}
\sup_{1\leq i\leq d}|g_i(z)|<\infty, \forall z\in \mathbb H.	\notag
\end{align}
Theorem 2.12 (i) in \cite{ajanki2015quadratic} implies Lemma \ref{QVEstab} holds with $\Lambda(z):=\displaystyle \sup_{1\leq i\leq d}|f_n^{(i)}(z)-g_i(z)|$ for any fixed $z\in D_{n,\varepsilon}$. Similar to the proof of Lemma \ref{qvelemma1}, we have 
\begin{align*}
|s_n(z)-m(z)|&=|\sum_{l=1}^d \frac{N_l}{n}f_n^{(l)}(z)-\sum_{l=1}^d \alpha_l g_l(z)|\\
&\leq |\sum_{l=1}^d \frac{N_l}{n}f_n^{(l)}(z)-\sum_{l=1}^d\alpha_l f_n^{(l)}(z)|+|\sum_{l=1}^d\alpha_l f_n^{(l)}(z)-\sum_{l=1}^d \alpha_l g_l(z)|\\
             &\leq \sum_{l=1}^d |(\frac{N_l}{n}-\alpha_l)f_n^{(l)}(z)|+\sum_{l=1}^d \alpha_l|f_n^{(l)}(z)- g_l(z)|=o(1)
\end{align*}
uniformly for all $z\in D_{n,\varepsilon}$ with probability at least $1-n^{-C}$. Hence the local law for  $\frac{A_n}{\sqrt {n}\sigma}$ is proved.
	\end{proof}

We have the corresponding infinity norm bound for eigenvectors in the bulk.
\begin{cor} \label{coreigenv}
Let $A_n$ be an adjacency matrix of a stochastic block model. For any bulk interval $I$ such that eigenvalue $\lambda_i(\frac{A_n}{\sqrt n\sigma })\in I$ and any constant $C_1>0$, with  probability at least $1-n^{-C_1}$, the corresponding unit eigenvector $u_i(\frac{A_n}{\sqrt n\sigma})$ satisfies 
\begin{align}\left\|u_i\left(\frac{A_n}{\sqrt n\sigma}\right)\right\|_{\infty}\leq \frac{C_2\sqrt{\log n}}{\sqrt {np}}.	\notag
\end{align}  for some constant $C_2>0$.
\end{cor}

\begin{proof} Let $W_n:=\frac{A_n}{\sqrt n \sigma}$. For any $\lambda_i(W_n)$ in the bulk, by Corollary \ref{SSBM}, one can find an interval $I$ centered at $\lambda_i(W_n)$ and $|I|=\frac{C_2\log n}{np}$ such that $N_I\geq \delta_1n |I|$ for some small $\delta_1>0$ with probability at least $1-n^{-C_1-3}$. We can find a set $J\subset \{1,\dots, n-1\}$ with $|J|\geq N_I/2$ such that $|\lambda_j(W_{n-1})-\lambda_i(W_n)|\leq |I|$ for all $j\in J$.  Let $X_k$ be the $k$-th column of $\frac{A_n}{\sigma}$ with the $k$-th entry removed, then $X_k=\sqrt n a_k$.

Since $X_k$ is not centered, we need  to show
\begin{align} \label{83}
\sum_{j\in J}|u_j(W_{n,k})^*X_k|^2=\|\pi_H(X_k)\|^2=\Omega(|J|)
\end{align} with probability at least $1-n^{-C_1-3}$, where $H$ is the subspace spanned by all orthonormal eigenvectors associated to eigenvalues $\lambda_j(W_{n,k}), j\in J$ and $\dim (H)=|J|$.

Let $H_1=H\cap H_2$, where $H_2$ is the subspace orthogonal to the vector $\mathbb Ea_k$. The dimension of $H_1$ is at least $|J|-1$. Let $b_k=a_k-\mathbb Ea_k$, then the entries of $b_k$ are centered with the same variances as $a_k$. By Lemma \ref{same}, we have 
\begin{align*}\|\pi_{H_1}(b_k)\|^2=\Omega\left(\frac{|J|}{n}\right)
\end{align*}  
with probability at least $1-n^{-C_1-3}$. Moreover,
\begin{align}\|\pi_{H}(a_k)\|=\|\pi_H(b_k+\mathbb Ea_k)\|\geq \|\pi_{H_1}(b_k+\mathbb Ea_k)\|=\|\pi_{H_1}(b_k)\|,\notag
\end{align}
which implies \eqref{83} holds.	The rest of the proof follows from the proof of Theorem \ref{deloc}.
\end{proof}

\subsubsection{Unbounded Number of Classes}
For the Stochastic Block Models,  if we allow the number of classes $d\to\infty$ as $n\to\infty$, a local law can be proved under the following assumptions
\begin{align} 
 d&=o\left(\frac{n}{g(n)}\right)\label{sum},\\
	 \sum_{i=1}^{d}\left|\frac{\sigma_{kl}^2}{\sigma^2}-c_{kl}\right| &=o\left(\frac{1}{g(n)}\right).\label{515}
\end{align}

We will compare the Stieltjes transform of the empirical spectral distribution to the measure whose Stieltjes transform satisfies the following equations:
\begin{align}
m_n(z)&=\sum_{i=1}^d \frac{N_i}{n} g_{n,i}(z)\label{qveint}\\
-\frac{1}{g_{n,i}(z)}&=z+\sum_{i=1}^d\frac{N_i}{n} c_{ij}g_{n,j}(z). 	\label{qveint2}
\end{align}

We have the following local law for SBM with unbounded number of blocks.
 \begin{cor}\label{corincrease}
Let $A_n$ be an adjacency matrix of SBM with assumptions \eqref{sum},\eqref{515}. Let $\rho_n$ be the probability measure corresponding to equations \eqref{qveint},\eqref{qveint2}. For any constants $\delta,C_1>0$, there exists a constant $C_2$ such that with probability at least $1-n^{-C_1}$ the following holds.
	 For any  bulk interval $I$ of length $\displaystyle |I|\geq  \frac{C_2\log n}{np}$,    the number of eigenvalues $N_I$ of $\displaystyle \frac{A_n}{\sqrt{n}\sigma}$ in $I$ obeys the concentration estimate
	\begin{align} 
	 \left |N_I-n\int_I \rho_n(x)dx \right|\leq \delta n |I|.	\end{align} \notag
	\end{cor}

\begin{proof}
Since $d=o\left(\frac{n}{g(n)}\right)$, recall the definition of $\tilde{A}_n$ from \eqref{tilde}, by Cauchy interlacing law,
\begin{align}|N_I(A_n)-N_I(A_n-\mathbb E(\tilde{A	}_n))|\leq d=o(n|I|).\notag
\end{align} 
 It suffices to prove the statement for the centered matrix
$\displaystyle W_n:=\frac{A_n-\mathbb E\tilde{A}_n}{\sqrt n\sigma}.$
The proof then follows from Corollary \ref{SSBM} with assumption \eqref{515}.

\end{proof}
\begin{remark}
	Different from Corollary \ref{SSBM}, in Corollary \ref{corincrease}, we are not comparing the empirical spectral distribution to a limiting spectral distribution $\rho$ independent of $n$. If we assume $ \displaystyle \frac{N_i}{n}\to\alpha_i, \alpha_1\geq \alpha_2\cdots \geq\cdots ,$ and $\displaystyle \sum_{i=1}^{\infty}\alpha_i=1$, one can show that $\rho_n$ converge to some $\rho$ (see Section 7 in \cite{zhu2018graphon} for further details).
But we don't have a local law comparing $N_I$ with $\displaystyle n\int_{I}\rho(x) dx$. In fact, let $S_n$ be the symmetric function on $[0,1]^2$ representing the variance profile as in \eqref{varianceprofile} and $S$ be its point-wise limit,  there is no upper bound for rate of convergence  on  $\displaystyle \sup_{x,y}|S_n(x,y)-S(x,y)|$. 
\end{remark}
\begin{remark}
With the same argument in the proof of Corollary \ref{coreigenv}, the infinity norm bound for eigenvectors in Corollary \ref{coreigenv} still holds for the SBM with unbounded number of classes.
\end{remark}

% If in two-column mode, this environment will change to single-column format so that long equations can be displayed. 
%\begin{widetext}
%$$\mbox{put long equation here}$$
%\end{widetext}

% Figures should be put into the text as floats. 
% Use the graphics or graphicx packages (distributed with LaTeX2e).
% Put the label that you will use with \ref{} command in the braces of the \label{} command.
%
% \begin{figure}
% \end{figure}

% \begin{table}
% \end{tabular}
% \end{table}

\section*{acknowledgments}
 The authors would like to thank PCMI Summer Session 2017 on Random Matrices, during which a part of this work was done. This work was supported by NSF DMS-1712630.


\bibliographystyle{plain}
\bibliography{ref.bib}

\begin{thebibliography}{10}

\bibitem{abbe2016exact}
Emmanuel Abbe, Afonso~S Bandeira, and Georgina Hall.
\newblock Exact recovery in the stochastic block model.
\newblock {\em IEEE Transactions on Information Theory}, 62(1):471--487, 2016.

\bibitem{abbe2015community}
Emmanuel Abbe and Colin Sandon.
\newblock Community detection in general stochastic block models: Fundamental
  limits and efficient algorithms for recovery.
\newblock In {\em Foundations of Computer Science (FOCS), 2015 IEEE 56th Annual
  Symposium on}, pages 670--688. IEEE, 2015.

\bibitem{adlam2015spectral}
Ben Adlam and Ziliang Che.
\newblock Spectral statistics of sparse random graphs with a general degree
  distribution.
\newblock 2015.
\newblock arXiv:1509.03368.

\bibitem{ajanki2015quadratic}
Oskari~H Ajanki, L{\'a}szl{\'o} Erd{\H{o}}s, and Torben Kr{\"u}ger.
\newblock Quadratic vector equations on complex upper half-plane.
\newblock {\em arXiv preprint arXiv:1506.05095}, 2015.

\bibitem{ajanki2015universality}
Oskari~H Ajanki, L{\'a}szl{\'o} Erd{\H{o}}s, and Torben Kr{\"u}ger.
\newblock Universality for general wigner-type matrices.
\newblock {\em Probability Theory and Related Fields}, pages 1--61, 2015.

\bibitem{avrachenkov2015spectral}
Konstantin Avrachenkov, Laura Cottatellucci, and Arun Kadavankandy.
\newblock Spectral properties of random matrices for stochastic block model.
\newblock In {\em Modeling and Optimization in Mobile, Ad Hoc, and Wireless
  Networks (WiOpt), 2015 13th International Symposium on}, pages 537--544.
  IEEE, 2015.

\bibitem{bai2010spectral}
Zhidong Bai and Jack~W Silverstein.
\newblock {\em Spectral analysis of large dimensional random matrices},
  volume~20.
\newblock Springer, 2010.

\bibitem{BGBK2}
Florent Benaych-Georges, Chalres Bordenave, and Antti Knowles.
\newblock Spectral radii of sparse random matrices.
\newblock {\em arXiv preprint arXiv:1704.02945}, 2017.

\bibitem{BGBK1}
Florent Benaych-Georges, Charles Bordenave, and Antti Knowles.
\newblock Largest eigenvalues of sparse inhomogeneous {E}rd{\H{o}}s-{R}{\'e}nyi
  graphs.
\newblock {\em arXiv preprint arXiv:1704.02953}, 2017.

\bibitem{brito2016recovery}
Gerandy Brito, Ioana Dumitriu, Shirshendu Ganguly, Christopher Hoffman, and
  Linh~V Tran.
\newblock Recovery and rigidity in a regular stochastic block model.
\newblock In {\em Proceedings of the Twenty-Seventh Annual ACM-SIAM Symposium
  on Discrete Algorithms}, pages 1589--1601. Society for Industrial and Applied
  Mathematics, 2016.

\bibitem{coja2010graph}
Amin Coja-Oghlan.
\newblock Graph partitioning via adaptive spectral techniques.
\newblock {\em Combinatorics, Probability and Computing}, 19(02):227--284,
  2010.

\bibitem{ding2014spectral}
Xue Ding.
\newblock Spectral analysis of large block random matrices with rectangular
  blocks.
\newblock {\em Lithuanian Mathematical Journal}, 54(2):115--126, 2014.

\bibitem{er_univer2}
L\'aszl\'o Erd\H{o}s, Antti Knowles, Horng-Tzer Yau, and Jun Yin.
\newblock Spectral statistics of {E}rd{\H{o}}s-{R}{\'e}nyi graphs {II}:
  Eigenvalue spacing and the extreme eigenvalues.
\newblock {\em Communications in Mathematical Physics}, 314(3):587--640, Sep
  2012.

\bibitem{er_univer1}
L\'aszl\'o Erd\H{o}s, Antti Knowles, Horng-Tzer Yau, and Jun Yin.
\newblock Spectral statistics of {E}rd{\H{o}}s-{R}{\'e}nyi graphs {I}: Local
  semicircle law.
\newblock {\em The Annals of Probability}, 41(3B):2279--2375, 05 2013.

\bibitem{erdos2009}
L\'aszl\'o Erd\H{o}s, Benjamin Schlein, and Horng-Tzer Yau.
\newblock Semicircle law on short scales and delocalization of eigenvectors for
  wigner random matrices.
\newblock {\em The Annals of Probability}, 37(3):815--852, 05 2009.

\bibitem{erdos2011universality}
L{\'a}szl{\'o} Erd\H{o}s, Horng-Tzer Yau, and Jun Yin.
\newblock Universality for generalized {W}igner matrices with {B}ernoulli
  distribution.
\newblock {\em Journal of Combinatorics}, 2(1):15--81, 2011.

\bibitem{erdHos2010bulk}
L{\'a}szl{\'o} Erd{\H{o}}s, Sandrine P{\'e}ch{\'e}, Jos{\'e}~A Ram{\'\i}rez,
  Benjamin Schlein, and Horng-Tzer Yau.
\newblock Bulk universality for wigner matrices.
\newblock {\em Communications on Pure and Applied Mathematics}, 63(7):895--925,
  2010.

\bibitem{erdos2012bulk}
L{\'a}szl{\'o} Erd{\H{o}}s, Horng-Tzer Yau, and Jun Yin.
\newblock Bulk universality for generalized wigner matrices.
\newblock {\em Probability Theory and Related Fields}, pages 1--67, 2012.

\bibitem{far2006spectra}
Reza~Rashidi Far, Tamer Oraby, Wlodzimierz Bryc, and Roland Speicher.
\newblock Spectra of large block matrices.
\newblock {\em arXiv preprint cs/0610045}, 2006.

\bibitem{girko2001theory}
Vi͡acheslav~Leonidovich Girko.
\newblock {\em Theory of stochastic canonical equations}, volume~2.
\newblock Springer Science \& Business Media, 2001.

\bibitem{he2018local}
Yukun He, Antti Knowles, and Matteo Marcozzi.
\newblock Local law and complete eigenvector delocalization for supercritical
  {E}rd{\H{o}}s-{R}{\'e}nyi graphs.
\newblock {\em arXiv preprint arXiv:1808.09437}, 2018.

\bibitem{HLL83}
Paul~W. Holland, Kathryn~Blackmond Laskey, and Samuel Leinhardt.
\newblock Stochastic blockmodels: First steps.
\newblock {\em Social Networks}, 5(2):109 -- 137, 1983.

\bibitem{er_laplacians}
Jiaoyang Huang and Benjamin Landon.
\newblock {Spectral statistics of sparse Erd\H{o}s-R\'enyi graph Laplacians}.
\newblock 2015.
\newblock arXiv:1510.06390v1.

\bibitem{huang2015bulk}
Jiaoyang Huang, Benjamin Landon, and Horng-Tzer Yau.
\newblock Bulk universality of sparse random matrices.
\newblock {\em Journal of Mathematical Physics}, 56(12):123301, 2015.

\bibitem{krzakala2013spectral}
Florent Krzakala, Cristopher Moore, Elchanan Mossel, Joe Neeman, Allan Sly,
  Lenka Zdeborov{\'a}, and Pan Zhang.
\newblock Spectral redemption in clustering sparse networks.
\newblock {\em Proceedings of the National Academy of Sciences},
  110(52):20935--20940, 2013.

\bibitem{luh2018sparse}
Kyle Luh and Van Vu.
\newblock Sparse random matrices have simple spectrum.
\newblock {\em arXiv preprint arXiv:1802.03662}, 2018.

\bibitem{o2016eigenvectors}
Sean O'Rourke, Van Vu, and Ke~Wang.
\newblock Eigenvectors of random matrices: a survey.
\newblock {\em Journal of Combinatorial Theory, Series A}, 144:361--442, 2016.

\bibitem{rudelson2013hanson}
Mark Rudelson, Roman Vershynin, et~al.
\newblock Hanson-wright inequality and sub-gaussian concentration.
\newblock {\em Electronic Communications in Probability}, 18, 2013.

\bibitem{shlyakhtenko1998gaussian}
Dimitri Shlyakhtenko.
\newblock Gaussian random band matrices and operator-valued free probability
  theory.
\newblock {\em Banach Center Publications}, 43(1):359--368, 1998.

\bibitem{tao2011random}
Terence Tao and Van Vu.
\newblock Random matrices: universality of local eigenvalue statistics.
\newblock {\em Acta mathematica}, 206(1):127--204, 2011.

\bibitem{tran2013sparse}
Linh~V Tran, Van~H Vu, and Ke~Wang.
\newblock Sparse random graphs: Eigenvalues and eigenvectors.
\newblock {\em Random Structures \& Algorithms}, 42(1):110--134, 2013.

\bibitem{vu2015random}
Van Vu and Ke~Wang.
\newblock Random weighted projections, random quadratic forms and random
  eigenvectors.
\newblock {\em Random Structures \& Algorithms}, 47(4):792--821, 2015.

\bibitem{wood2012universality}
Philip~Matchett Wood.
\newblock Universality and the circular law for sparse random matrices.
\newblock {\em The Annals of Applied Probability}, 22(3):1266--1300, 2012.

\bibitem{zhu2018graphon}
Yizhe Zhu.
\newblock A graphon approach to limiting spectral distributions of wigner-type
  matrices.
\newblock {\em arXiv preprint arXiv:1806.11246}, 2018.

\end{thebibliography}

\end{document}